\documentclass[12pt,leqno]{article}
\usepackage{amsfonts,amsthm,amsmath}
\newtheorem{thm}{Theorem}
\newtheorem{prop}[thm]{Proposition}
\newtheorem{lem}[thm]{Lemma}
\newtheorem{cor}[thm]{Corollary}
\theoremstyle{remark}
\newtheorem{rem}[thm]{Remark}
\newtheorem{Q}[thm]{Question}
\newcommand{\FF}{\mathbb{F}}

\DeclareMathOperator{\wt}{wt}

\begin{document}

\title{On the performance of optimal
double circulant even codes}

\author{
T. Aaron Gulliver\thanks{Department of Electrical and Computer Engineering,
University of Victoria,
P.O. Box 1700, STN CSC, Victoria, BC,
Canada V8W 2Y2.
email: agullive@ece.uvic.ca}
\text{ and }
Masaaki Harada\thanks{
Research Center for Pure and Applied Mathematics,
Graduate School of Information Sciences,
Tohoku University, Sendai 980--8579, Japan.
email: mharada@m.tohoku.ac.jp.}
}

\maketitle

\begin{abstract}
In this note, we investigate the performance of
optimal double circulant even codes which are not self-dual,
as measured by the decoding error probability in bounded distance decoding.
To do this, we classify the optimal double circulant even codes
that are not self-dual which have the smallest weight distribution
for lengths up to $72$.
We also give some restrictions on the weight distributions of
(extremal) self-dual $[54,27,10]$ codes
with shadows of minimum weight $3$.
Finally, we consider the performance of extremal self-dual codes
of lengths $88$ and $112$.
\end{abstract}

\section{Introduction}
\label{Sec:Intro}
A (binary) $[n,k]$ {\em code} $C$ is a $k$-dimensional vector subspace
of $\FF_2^n$,
where $\FF_2$ denotes the finite field of order $2$.
All codes in this note are binary.
The parameter $n$ is called the {\em length} of $C$.
The {\em weight} $\wt(x)$ of a vector $x \in \FF_2^n$ is
the number of non-zero components of $x$.
A vector of $C$ is called a {\em codeword}.
The minimum non-zero weight of all codewords in $C$ is called
the {\em minimum weight} of $C$ and an $[n,k]$ code with minimum
weight $d$ is called an $[n,k,d]$ code.
A code with only even weights is called {\em even}.
Two codes are {\em equivalent} if one can be obtained from the other by a
permutation of coordinates.

Let $C$ be an $[n,k,d]$ code.
Throughout this note,
let $A_i$ denote the number of codewords of weight $i$ in $C$.
The  sequence $(A_0,A_1,\ldots,A_n)$ is called the {\em weight distribution} of $C$.
A code $C$ of length $n$ is said to be
{\em formally self-dual\/} if $C$ and $C^\perp$
have identical weight distributions, where
$C^\perp$ is the dual code of $C$.
A code $C$ is {\em isodual\/} if $C$ and $C^\perp$ are equivalent, and
$C$ is {\em self-dual\/} if $C=C^\perp$.
A self-dual code is an even isodual code, and an isodual code is a
formally self-dual code.
There are formally self-dual even codes
which are not self-dual.
One reason for our interest in formally self-dual even codes
is that for some lengths there are formally self-dual even codes
with larger minimum weights than any self-dual code of that length.
Double circulant codes are a remarkable class of isodual codes.


The question of decoding error probabilities was studied
by Faldum, Lafuente, Ochoa and Willems~\cite{FLOW}
for bounded distance decoding.
Let $C$ and $C'$ be $[n,k,d]$ codes with weight distributions
$(A_0,A_1,\ldots,A_n)$ and $(A'_0,A'_1,\ldots,A'_n)$, respectively.
Suppose that symbol errors are independent and the symbol error probability is small.
Then $C$  has a smaller decoding error probability than $C'$
if and only if
\begin{equation}\label{eq:WD}
(A_0,A_1,\ldots,A_n) \prec (A'_0,A'_1,\ldots,A'_n),
\end{equation}
where $\prec$ means the lexicographic order, that is,
there is an integer $s \in \{0,1,\ldots,n\}$ such that
$A_i=A'_i$ for all $i < s$ but $A_s < A'_s$~\cite[Theorem~3.4]{FLOW}.
We say that $C$ {\em performs better} than $C'$ if~\eqref{eq:WD} holds.
By making use of~\cite[Theorem~3.4]{FLOW},
Bouyuklieva, Malevich and Willems~\cite{BMW}
investigated and compared the performance of
extremal doubly even and singly even self-dual codes.

In this note, we consider the performance of
optimal double circulant even codes which are not self-dual using~\cite[Theorem~3.4]{FLOW}.
To do this, we classify the optimal double circulant even codes
that are not self-dual which have the smallest weight distribution
for lengths up to $72$.
For $(2n,d)= (32,8)$, $(36,8)$, $(38,8)$, $(40,8)$, $(46,10)$,
$(52,10)$, $(56,12)$, $(60,12)$, $(62,12)$, $(64,12)$,
$(66,12)$ and $(68,12)$,
we demonstrate that there is an optimal double circulant even
$[2n,n,d]$ code $C$ which is not self-dual
such that $C$ performs better than any self-dual $[2n,n,d]$ code.
We also give some restrictions on weight distributions of
(extremal) self-dual $[54,27,10]$ codes
having shadows of minimum weight $3$.
Finally, we consider the performance of extremal self-dual codes
of lengths $88$ and $112$.

\section{Double circulant codes}

An $n \times n$ circulant matrix has the form:
\[
\left(
\begin{array}{ccccc}
r_1&r_2&r_3& \cdots &r_{n} \\
r_{n}&r_1&r_2& \cdots &r_{n-1} \\
\vdots &\vdots & \vdots && \vdots\\
r_2&r_3&r_4& \cdots&r_1
\end{array}
\right)
\]
so that each successive row is a cyclic shift of the previous one.
A {\em pure double circulant} code and a {\em bordered
double circulant} code have generator matrices of the form:
\begin{equation}\label{eq:pDCC}
\left(\begin{array}{ccccc}
{} & I_n  & {} & R & {} \\
\end{array}\right)
\end{equation}
and
\begin{equation}\label{eq:bDCC}
\left(\begin{array}{ccccccccc}
{} & {} & {}      & {} & {} & \alpha & 1  & \cdots &1  \\
{} & {} & {}      & {} & {} & 1     & {} & {}     &{} \\
{} & {} & I_{n} & {} & {} &\vdots & {} & R'     &{} \\
{} & {} & {}      & {} & {} & 1     & {} &{}      &{} \\
\end{array}\right),
\end{equation}
respectively,
where $I_n$ is the identity matrix of order $n$,
$R$ (resp.\ $R'$) is an $n \times n$ (resp.\ $n-1 \times n-1$)
circulant matrix, and $\alpha \in \FF_2$.
These two families are called {\em double circulant} codes.
Since we consider only even codes in this note,
$\alpha=0$ if $n$ is even and $\alpha=1$ if $n$ is odd.

It is a fundamental problem to classify double circulant codes over
binary and nonbinary fields as well as finite rings,
for modest lengths, up to equivalence.
There has been significant research on finding double circulant self-dual codes with the
large minimum weights (see e.g., \cite{GH-DCC}, \cite{GH-DCC88}, \cite{GH-DCC90},
\cite{HGK-DCC}).
Beyond self-dual codes,
few results on the classification of double circulant codes
are known (see e.g., \cite{G-H-DCC}).
One reason for this is that the classification of double circulant
codes which are not self-dual is much more difficult than
double circulant self-dual codes.

In this note, we consider codes $C$ satisfying
the following conditions:
\begin{itemize}
\item[(C1)]
$C$ is a pure (resp.\ bordered) double circulant even code of length $2n$ which is not self-dual.
\item[(C2)]
$C$ has the largest minimum weight $d_{P}$ (resp.\ $d_{B}$)
among pure (resp.\ bordered) double circulant codes of
length $2n$ which are not self-dual.
\item[(C3)]
$C$ has the smallest weight distribution
$(A_0,A_1,\ldots,A_n)$ under the lexicographic order $\prec$
among
pure (resp.\ bordered) double circulant even codes of length $2n$ and
minimum weight $d_{P}$ (resp.\ $d_{B}$) which are not self-dual.
\end{itemize}

We say that a double circulant even code which is not self-dual
is {\em optimal} if it has the largest minimum weight among all
double circulant even codes of that length which are not self-dual.

The following lemma is trivial.

\begin{lem}
\label{lem:wt}
If $C$ is a pure double circulant even code of length $2n$
with generator matrix~\eqref{eq:pDCC},
then every row of $R$ has odd weight.
If $C$ is a bordered double circulant even code of length $2n$
with generator matrix~\eqref{eq:bDCC},
then every row of $R'$ has even weight.
\end{lem}

\section{Weight enumerators}

The weight enumerator $W$ of a formally self-dual even code of length $2n$
can be represented as an integral combination of Gleason polynomials
(see~\cite{MS73}), so that
\begin{equation}\label{eq:WE}
W= \sum_{j=0}^{\lfloor n/4 \rfloor}a_j(1+y^2)^{n-4j}\{y^2(1-y^2)^2\}^j,
\end{equation}
for some integers $a_j$ with $a_0=1$.
Note that the weight enumerators of
formally self-dual even codes of length $2n$ as well as
self-dual codes of length $2n$ can be expressed using~\eqref{eq:WE}.
This is one of the reasons why we compare the performance of optimal double circulant even codes
which are not self-dual with self-dual codes having the largest minimum weight.

For the following parameters
\begin{equation*}\label{eq:para}
\begin{array}{ll}
(2n,d)=&
(32, 8),
(34, 8),
(36, 8),
(38, 8),
(40, 8),
(42,10),
(44,10),
(46,10),
\\ &
(48,10),
(50,10),
(52,10),
(54,10),
(56,12),
(58,12),
(60,12),
\\ &
(62,12),
(64,12),
(66,12),
(68,12),
(70,12),
(72,14),
\end{array}
\end{equation*}
the possible weight enumerators $W_{2n,d}=\sum_{i=0}^{2n}A_i y^i$ of
a formally self-dual even $[2n,n,d]$ code
can be determined as follows.
The coefficients $A_i$ ($i=0,d,d+2,d+4,d+6$)
are listed in Table~\ref{Tab:WE},
where $a,b,c$ are integers.
For $(2n,d)=(32,8)$, $(34,8)$, $(36,8)$, $(38,8)$,
$(42,10)$, $(44,10)$ and $(46,10)$,
the weight enumerator $W_{2n,d}$ is
completely determined by only $A_d$.
For $(2n,d)=(40,8)$, $(48,10)$, $(50,10)$, $(52,10)$, $(54,10)$,
$(56,12)$, $(56,12)$, $(58,12)$, $(60,12)$ and $(62,12)$,
the weight enumerator $W_{2n,d}$ is
completely determined by $A_d$ and $A_{d+2}$.
For the remaining values of $(2n,d)$,
the weight enumerator $W_{2n,d}$ is completely determined
by $A_d$, $A_{d+2}$ and $A_{d+4}$.

\begin{table}[thb]
\caption{Possible weight enumerators $W_{2n,d}$}
\label{Tab:WE}
\begin{center}
{\footnotesize
\begin{tabular}{c|c|c|c|c|c}
\noalign{\hrule height1pt}
$(2n,d)$ & $A_0$ & $A_d$ & $A_{d+2}$ & $A_{d+4}$ & $A_{d+6}$ \\
\hline
$(32, 8)$ & 1 & $a$ & $4960-8a$&$-3472+28a$&$34720-56a$\\
$(34, 8)$ & 1 & $a$ & $4114-7a$&$2516+20a$&$29172-28a$\\
$(36, 8)$ & 1 & $a$ & $3366-6a$&$6630+13a$&$30600-8a$\\
$(38, 8)$ & 1 & $a$ & $2717-5a$&$9177+7a$&$35910+5a$\\
$(40, 8)$ & 1 & $a$ & $-4a+b$ & $32110+2a-10b$ & $-54720+12a+45b$\\
$(42,10)$ & 1 & $a$ & $26117-9a$&$-10455+35a$&$286713-75a$\\
$(44,10)$ & 1 & $a$ & $21021-8a$&$19712+26a$&$250778-40a$\\
$(46,10)$ & 1 & $a$ & $16744-7a$&$38709+18a$&$249458-14a$\\
$(48,10)$ & 1 & $a$ & $-6a+b$ & $207552+11a-12b$ & $-606441+4a+66b$\\
$(50,10)$ & 1 & $a$ & $-5a+b$ & $166600+5a-11b$ & $-271950+15a+54b$\\
$(52,10)$ & 1 & $a$ & $-4a+b$ & $132600-10b$ & $-41990+20a+43b$\\
$(54,10)$ & 1 & $a$ & $-3a+b$ & $104652-4a-9b$ & $107406+20a+33b$\\
$(56,12)$ & 1 & $a$ & $-8a+b$ & $1343034+24a-14b$ & $-5765760-24a+91b$\\
$(58,12)$ & 1 & $a$ & $-7a+b$ & $1067838+16a-13b$ & $-3224452+77b$\\
$(60,12)$ & 1 & $a$ & $-6a+b$ & $843030+9a-12b$ & $-1454640+16a+64b$\\
$(62,12)$ & 1 & $a$ & $-5a+b$ & $660858+3a-11b$ & $-270940+25a+52b$\\
$(64,12)$ & 1 & $a$ & $-4a+b$ & $-2a-10b+c$ & $8707776+28a+41b-16c$\\
$(66,12)$ & 1 & $a$ & $-3a+b$ & $-6a-9b+c$ & $6874010+26a+31b-15c$\\
$(68,12)$ & 1 & $a$ & $-2a+b$ & $-9a-8b+c$ & $5393454+20a+22b-14c$\\
$(70,12)$ & 1 & $a$ & $-a+b$ & $-11a-7b+c$ & $4206125+11a+14b-13c$\\
$(72,14)$ & 1 & $a$ & $-6a+b$ & $7a-12b+c$ & $56583450+28a+62b-18c$\\
\noalign{\hrule height1pt}
\end{tabular}
}
\end{center}
\end{table}

\section{Performance of double circulant even codes}

A classification of optimal double circulant even codes
was given in~\cite{G-H-DCC} for lengths up to $30$.
For lengths $2n$ with $32 \le 2n \le 72$,
by determining the largest minimum weights $d_{P}$ (resp.\ $d_{B}$),
our exhaustive search found all distinct
pure (resp.\ bordered) double circulant even codes
satisfying conditions (C1)--(C3).
This was done by considering all
$n \times n$ circulant matrices $R$ in~\eqref{eq:pDCC}
(resp.\ $n-1 \times n-1$ circulant matrices $R'$ in~\eqref{eq:bDCC})
satisfying
the condition given in Lemma~\ref{lem:wt}.
Since a cyclic shift of the first row for a
code defines an equivalent code,
the elimination of cyclic shifts
substantially reduced the number of codes
which had to be checked further for equivalence to
complete the classification.
Then {\sc Magma}~\cite{Magma} was employed to determine code
equivalence which completed the classification.

In Table~\ref{Tab:Res}, we list the numbers $N_{P}$
and $N_{B}$ of inequivalent pure and bordered
double circulant even codes
satisfying conditions (C1)--(C3), respectively.
In Table~\ref{Tab:Res}, we also list
$d_{P}$,  $A_{d_{P}}$,
$d_{B}$ and $A_{d_{B}}$.
For the pure and bordered
double circulant even codes
satisfying conditions (C1)--(C3),
the first rows of $R$ in~\eqref{eq:pDCC}
and $R'$ in~\eqref{eq:bDCC}
are listed in Tables~\ref{Tab:P} and~\ref{Tab:B},
respectively.
The minimum weights $d$ and $(A_d,A_{d+2},A_{d+4})$
are also listed in the tables.

\begin{table}[thb]
\caption{Double circulant even codes satisfying  (C1)--(C3)}
\label{Tab:Res}
\begin{center}
{\footnotesize
\begin{tabular}{c|c|c|c||c|c|c||c|cc}
\noalign{\hrule height1pt}
$2n$
& $d_{P}$  & $A_{d_{P}}$ & $N_{P}$
& $d_{B}$  & $A_{d_{B}}$ & $N_{B}$
& $d_{SD}$  & \multicolumn{2}{c}{$A_{d_{SD}}$}  \\
\hline
32&  8& $348$  &  2&  8 &$ 300$&   1&8&  364 &\cite{C-S}\\
34&  8& $272$  & 15&  8 &$ 272$&  10&6&  \multicolumn{2}{c}{-}\\
36&  8& $153$  &  4&  8 &$ 153$&   3&8&  225 &\cite{C-S}\\
38&  8& $ 76$  &  1&  8 &$  72$&   1&8&  171 &\cite{C-S}\\
40&  8& $ 25$  &  1&  8 &$  38$&   2&8&  125 &\cite{C-S}\\
42& 10& $1680$ &  2& 10 &$1682$&  1&8&  \multicolumn{2}{c}{-}\\
44& 10& $1144$ &  1& 10 &$1267$&  3&8&  \multicolumn{2}{c}{-}\\
46& 10& $851$  &  1& 10 &$ 858$&  2&10&  1012 &\cite{C-S}\\
48& 10& $480$  &  1& 10 &$ 575$&   1&12& 17296 &\cite{C-S}\\
50& 10& $325$  &  1& 10 &$ 356$&  1&10&  196 &\cite{HT}\\
52& 10& $156$  &  1& 10 &$ 150$&  1&10&  250 &\cite{C-S}\\
54& 10& $ 27$  &  1& 10 &$  52$&1&10&  7-135 &\cite{BO}, \cite{C-S}\\
56& 12& $4060$ &  1& 10 &$   3$&1&12&  4606-8190 &\cite{C-S}\\
58 &12& $3161$ &  1& 12 &$3227$&1&10&  \multicolumn{2}{c}{-}\\
60 &12& $2095$ &  1& 12 &$2146$&1&12&  2555 &\cite{RY07} \\
62 &12& $1333$ &  1 &12&$1290$&1&12&  1860 &\cite{DH} \\
64 &12& $544$  &  1 &12& $806$ &1 &12&  1312 &\cite{CHK64} \\
66 &12& $374$  &  1 &12&$480$&1&12& 858 &\cite{C-S} (see \cite{DGH}) \\
68 &12& $136$  &  1 &12&$165$&1&12& 442-486 &\cite{DGH}, \cite{YLGI} \\
70 &12& $35$   &  1 &14&12172&1&12-14&  \multicolumn{2}{c}{-}\\
72 &14&$8064$&1 &14&$8190$&1&12-16& \multicolumn{2}{c}{-}\\
\noalign{\hrule height1pt}
\end{tabular}
}
\end{center}
\end{table}

To compare the performance of the
optimal double circulant even codes which are not self-dual
as measured by the decoding error probability with bounded distance decoding,
we list in Table~\ref{Tab:Res} the largest minimum weight $d_{SD}$
among self-dual codes of length $2n$.
We also list the smallest number $A_{SD}$ of codewords of
minimum weight $d_{SD}$ among
self-dual codes of length $2n$ and minimum weight $d_{SD}$,
when $d_P=d_{SD}$ or $d_B=d_{SD}$, along with the references.
%
From Table~\ref{Tab:Res}, we have the following results concerning the
performance of optimal double circulant even codes which are not self-dual.

\begin{thm}\label{thm}
Suppose that
\begin{align*}
(2n,d)=& (32,8), (36,8), (38,8), (40,8), (46,10), (52,10), \\
& (56,12), (60,12), (62,12), (64,12), (66,12), (68,12).
\end{align*}
Then there is an optimal double circulant even
$[2n,n,d]$ code $C$ which is not self-dual
such that $C$ performs better than any self-dual $[2n,n,d]$ code.
\end{thm}

\begin{rem} \hfill
\label{rem}
\begin{itemize}
\item[\rm (1)]
For $2n=34, 42, 44, 58$, there is a double circulant even code $C$ of
length $2n$ which is not self-dual such that
$C$ has a larger minimum weight than any self-dual code of length $2n$.

\item[\rm (2)]
Up to equivalence,
there is a unique extremal doubly even self-dual $[48,24,12]$
code~\cite{HLTP}.
There is no double circulant even $[48,24,d]$ code with $d \ge 12$
which is not self-dual.

\item[\rm (3)]
There is a self-dual $[50,25,10]$ code $C$
such that $C$ performs better than any
double circulant even $[50,25,10]$ code which is not self-dual~\cite{HT}.

\item[\rm (4)]
There is a double circulant even $[70,35,14]$ code.
The largest minimum weight among currently known self-dual
codes of length $70$ is $12$.
The weight enumerator of an extremal
self-dual $[70,35,14]$ code $C$ is uniquely determined
and the number of codewords of weight $14$ is $11730$~\cite{DH99}.
If there is such a code $C$, then $C$ performs better than any
double circulant even $[70,35,14]$ code which is not self-dual.

\item[\rm (5)]
There is a double circulant even $[72,36,14]$ code.
The largest minimum weight among currently known self-dual
codes of length $72$ is $12$.
The existence of an extremal doubly even self-dual
$[72,36,16]$ code is a long-standing open question~\cite{Sloane72}
(see also \cite[Section~12]{RS-Handbook}).
\end{itemize}
\end{rem}

At the end of this section, we examine the equivalence of some codes
in Tables~\ref{Tab:P} and \ref{Tab:B}.
We verified by {\sc Magma}~\cite{Magma} that
there is no pair of equivalent codes among
$P_{34,i}$ $(i=1,2,\ldots,15)$ and
$B_{34,i}$ $(i=1,2,\ldots,10)$,
and that there is no pair of equivalent codes among
$P_{36,i}$ $(i=1,2,3,4)$ and
$B_{36,i}$ $(i=1,2,3)$.
A formally self-dual even $[70,35,14]$ code $C_{70}$ can be found
in~\cite{GH70}.
The code $C_{70}$ is equivalent to some bordered double circulant
codes, and $C_{70}$ and $B_{70}$ have identical weight
enumerators~\cite{GH70}.
Hence, $B_{70}$ must be equivalent to $C_{70}$.

\begin{table}[thbp]
\caption{Pure double circulant even codes satisfying  (C1)--(C3)}
\label{Tab:P}
\begin{center}
{\footnotesize
\begin{tabular}{c|r|c|c}
\noalign{\hrule height1pt}
Code &  \multicolumn{1}{c|}{First row} & $d$ & $(A_d,A_{d+2},A_{d+4})$\\
\hline
$P_{32,1}$&(1100101100110101)& 8&$(348,2176,6272)$ \\
$P_{32,2}$&(1110110100010011)& 8&$(348,2176,6272)$ \\
$P_{34,1}$&(11111110001000100)& 8&$(272,2210,7956)$ \\
$P_{34,2}$&(11100000111010110)& 8&$(272,2210,7956)$ \\
$P_{34,3}$&(11110101101101100)& 8&$(272,2210,7956)$ \\
$P_{34,4}$&(11110011101101010)& 8&$(272,2210,7956)$ \\
$P_{34,5}$&(10001011101100000)& 8&$(272,2210,7956)$ \\
$P_{34,6}$&(10001100110010100)& 8&$(272,2210,7956)$ \\
$P_{34,7}$&(11101110110100000)& 8&$(272,2210,7956)$ \\
$P_{34,8}$&(10100101100011110)& 8&$(272,2210,7956)$ \\
$P_{34,9}$&(10100100110010001)& 8&$(272,2210,7956)$ \\
$P_{34,10}$&(10101010011111000)& 8&$(272,2210,7956)$ \\
$P_{34,11}$&(10001110000100110)& 8&$(272,2210,7956)$ \\
$P_{34,12}$&(11010010010001111)& 8&$(272,2210,7956)$ \\
$P_{34,13}$&(10001100001110100)& 8&$(272,2210,7956)$ \\
$P_{34,14}$&(11011010100001101)& 8&$(272,2210,7956)$ \\
$P_{34,15}$&(11100001101010011)& 8&$(272,2210,7956)$ \\
$P_{36,1}$&(101011110110000001)& 8&$(153,2448,8619)$ \\
$P_{36,2}$&(111100001000010111)& 8&$(153,2448,8619)$ \\
$P_{36,3}$&(100110111010010001)& 8&$(153,2448,8619)$ \\
$P_{36,4}$&(100001010110111100)& 8&$(153,2448,8619)$ \\
$P_{38}$&(1111000001001010110)&  8&$(76,2337,9709)$ \\
$P_{40}$&(10101101111101111000)& 8&$(25,2080,10360)$ \\
$P_{42,1}$&(100001101101110010110)& 10&$(1680,10997,48345)$ \\
$P_{42,2}$&(101010010101110110111)& 10&$(1680,10997,48345)$ \\
$P_{44}$&(1001111111001011011011)&  10&$(1144,11869,49456)$ \\
$P_{46}$&(11001011010111100000001)& 10&$(851,10787,54027)$ \\
$P_{48}$&(110111000101111101110100)& 10&$(480,10384,53664)$ \\
$P_{50}$ & (1000100001011001001011101) & 10&$(325,8650,55200)$ \\
$P_{52}$ & (10001010100011011011000001) & 10&$(156,7267,53690)$ \\
$P_{54}$ & (111000000011101101100010011) & 10&$(27,6030,49545)$ \\
$P_{56}$ & (1001100011110101110111110100) & 12&$(4060,49420,293874)$ \\
$P_{58}$ & (11011000010100000000110011010) & 12&$(3161,41412,292407)$ \\
$P_{60}$ & (100000101101110000100111010001) & 12&$(2095,37320,263205)$ \\
$P_{62}$ & (0010100111101100111111010000000) & 12&$(1333, 30597, 254975)$ \\
$P_{64}$ & (10101000110010111100110100000000) & 12&$(544, 34304, 115756)$ \\
$P_{66}$ & (100100010010000101111011100100000) & 12&$(374, 20163, 203808)$ \\
$P_{68}$ & (1001001011010110101010101011000000) & 12&$(136, 15606, 176936)$ \\
$P_{70}$ & (01011011100110100101110000110000000) & 12&$(35, 11550, 151130)$ \\
$P_{72}$ & (101101101101001101001101111100010000) & 14 &$(8064, 127809, 1202464)$ \\
\noalign{\hrule height1pt}
\end{tabular}
}
\end{center}
\end{table}

\begin{table}[thbp]
\caption{Bordered double circulant even codes satisfying  (C1)--(C3)}
\label{Tab:B}
\begin{center}
{\footnotesize
\begin{tabular}{c|r|c|c}
\noalign{\hrule height1pt}
Code &  \multicolumn{1}{c|}{First row} & $d$ & $(A_d,A_{d+2},A_{d+4})$\\
\hline
$B_{32}$&(100101010001111) &8&$(300,2560,4928)$\\
$B_{34,1}$&(1001101010001101) &8&$(272,2210,7956)$\\
$B_{34,2}$&(1110111100010110) &8&$(272,2210,7956)$\\
$B_{34,3}$&(1010100111011101) &8&$(272,2210,7956)$\\
$B_{34,4}$&(1000110111011110) &8&$(272,2210,7956)$\\
$B_{34,5}$&(1110010011010001) &8&$(272,2210,7956)$\\
$B_{34,6}$&(1101101100101000) &8&$(272,2210,7956)$\\
$B_{34,7}$&(1001001100111010) &8&$(272,2210,7956)$\\
$B_{34,8}$&(1110000111110110) &8&$(272,2210,7956)$\\
$B_{34,9}$&(1110000111011110) &8&$(272,2210,7956)$\\
$B_{34,10}$&(1001010011010011) &8&$(272,2210,7956)$\\
$B_{36,1}$&(11001011010011101) &8&$(153,2448,8619)$\\
$B_{36,2}$&(11011100001010111) &8&$(153,2448,8619)$\\
$B_{36,3}$&(10001000101011011) &8&$(153,2448,8619)$\\
$B_{38}$&(110000101101101000) &8&$(72,2357,9681)$\\
$B_{40,1}$&(1100000111101000100) &8&$(38,2014,10526)$\\
$B_{40,2}$&(1010011001110001110) &8&$(38,2014,10526)$\\
$B_{42}$&(10011111001111010010) &10&$(1682,10979,48415)$\\
$B_{44,1}$&(101010000011101100110) &10&$(1267,10885,52654)$\\
$B_{44,2}$&(111100011011101010111) &10&$(1267,10885,52654)$\\
$B_{44,3}$&(110000111111101101101) &10&$(1267,10885,52654)$\\
$B_{46,1}$&(1110100010011100011000) &10&$(858,10738,54153)$\\
$B_{46,2}$&(1111100111111001000101) &10&$(858,10738,54153)$\\
$B_{48}$&(11010101000010011100010) &10&$(575,9752,55453)$\\
$B_{50}$ & (111110011001100111100010) &10&$(356,8524,55036)$\\
$B_{52}$ & (1010001000101001100100101) &10&$(150,7375,52850)$\\
$B_{54}$ & (11101011011000000010001110) &10&$(52,5876,50156)$\\
$B_{56}$ & (100111100001001000000100011) &10&$(3,4545,45477)$\\
$B_{58}$ & (1101101000010100111100110111) &12&$(3227,40950,293463)$\\
$B_{60}$ & (11001101111100101010111101100) &12&$(2146,36163,273876)$\\
$B_{62}$ & (110010100011110110110000000000) &12&$(1290, 30850, 254428)$\\
$B_{64}$ & (1000010101011010011011010000000) &12&$(806, 25358, 226982)$\\
$B_{66}$ & (10101110111101100111111011010000) &12&$(480, 19848, 203112)$\\
$B_{68}$ & (100011110101110110010101010100000) &12&$(165, 15620, 176099)$\\
$B_{70}$ & (1101000101110100101011110000000000) &14&$(12172, 147390, 1352811)$\\
$B_{72}$ & (10011110101111100101111001110111000) &14&$(8190, 126952, 1204560)$\\
\noalign{\hrule height1pt}
\end{tabular}
}
\end{center}
\end{table}

\section{Self-dual codes of length 54}

In this section, we consider the remaining case.
Table~\ref{Tab:Res} gives rise to a natural question,
namely, is there a self-dual $[54,27,10]$ code such that
$A_{10} < 27$?
A self-dual $[54,27,10]$ code $C$ and
its shadow $S$ (see~\cite{C-S} for the definition of shadows),
have the following possible weight enumerators $W_i(C)$ and $W_i(S)$,
respectively $(i=1,2)$:
\begin{align*}
&\left\{
\begin{array}{lll}
W_1(C) &=& 1
+ (351  - 8 \beta )y^{10}
+ (5031  + 24 \beta )y^{12}
+ \cdots, \\
W_1(S) &=&
\beta y^7
+ (2808 - 10 \beta )y^{11}
+ \cdots,
\end{array}\right. \\
&\left\{
\begin{array}{lll}
W_2(C) &=&
1
+ (351 - 8 \beta )y^{10}
+ (5543  + 24 \beta )y^{12}
+ \cdots,\\
W_2(S) &=&
y^3
+(- 12 + \beta) y^7
+( 2874 - 10 \beta) y^{11}
 + \cdots,\\
\end{array}\right.
\end{align*}
where $\beta$ is an integer with $0 \le \beta \le 43$ for $i=1$
and $12 \le \beta \le 43$ for $i=2$~\cite{C-S}.
Self-dual codes exist with $W_1(C)$ for $\beta=0,1,\ldots, 20,22,26$ and
self-dual codes exist with $W_2(C)$ for $\beta=12,13,\ldots, 22,24,26,27$
(see~\cite{BO}, \cite{C-S}, \cite{Hu05},
\cite{YL14}, \cite{YR11}).
Note that the smallest number $A_{10}$
among currently known self-dual codes of length $54$
and minimum weight $10$ is $135$.

\begin{prop}
If there is a self-dual $[54,27,10]$ code $C$ such that
$A_{10} < 27$, then $C$ has weight enumerator $W_1(C)$ with
$41 \le \beta \le 43$.
\end{prop}
\begin{proof}
Let $C$  be a self-dual $[54,27,10]$ code such that
$A_{10} < 27$.
Suppose that $C$ has weight enumerator $W_2(C)$.
From the assumption, $\beta > 12$.  Thus, there is a vector of weight $7$
in the shadow $S$.
Let $x_1$ and $x_2$ be vectors of weights $3$ and $7$ in $S$,
respectively.
Since the sum of two distinct vectors of $S$ is a codeword of $C$,
$x_1+x_2$ must be a codeword of weight $10$.
From the coefficient of $y^7$ in $W_2(S)$, we have
\begin{align*}
351-8\beta \ge& \beta -12 \\
40 \ge & \beta.
\end{align*}
Hence, we have that $A_{10} \ge  31$,
which is a contradiction.
It follows from $A_{10} < 27$ that
$41 \le \beta \le 43$ in $W_1(C)$.
\end{proof}

Hence, the above question can be refined as follows.

\begin{Q}
Is there a  self-dual $[54,27,10]$ code
which has weight enumerator $W_1(C)$ with $41 \le \beta \le 43$?
\end{Q}

As a consequence of the proof of the above proposition,
we have the following.

\begin{cor}
If there is a self-dual $[54,27,10]$ code with weight enumerator $W_2(C)$,
then $\beta \in \{12,13,\ldots,40\}$.
\end{cor}

\section{Performance of extremal self-dual codes of lengths 88 and 112}

The performance of extremal doubly even and singly even self-dual codes was compared
in \cite[Section 4]{BMW} for lengths $24k+8$ $(k=1,2,3,4)$
and $24k+16$ $(k=1,2)$.
In this section, we consider the case for lengths $24k+16$ $(k=3,4)$.

An extremal doubly even self-dual code of length $88$ has
the following weight enumerator (see~\cite{MS73}):
\[
1 + 32164 y^{16} + 6992832 y^{20} + 535731625 y^{24} + \cdots.
\]
There are at least $470$ inequivalent
extremal doubly even self-dual codes of length $88$
(see~\cite{GH-DCC88}).
An extremal singly even self-dual code of length $88$
was given in~\cite{HN} which has the following weight enumerator:
\[
1 + 18436 y^{16} + 268928 y^{18} + 3493248 y^{20}
+ 267717065 y^{24} +\cdots.
\]
Hence, there is an extremal singly even self-dual code $C$ of
length $88$ such that $C$ performs better than any
extremal doubly even self-dual code of that length.

An extremal doubly even self-dual code of length $112$ has
the following weight enumerator (see~\cite{MS73}):
\[
1 + 355740 y^{20} + 95307030 y^{24} + 10847290300 y^{28} + \cdots.
\]
An extremal doubly even self-dual code of length $112$
was found in~\cite{H112}.
By \cite[Theorem~5]{C-S}, the possible weight enumerators
$W_i(C)$ and $W_i(S)$
of an extremal singly even self-dual code $C$ of length $112$
and its shadow $S$ are:
\begin{align*}
&\left\{
\begin{array}{lll}
W_1(C) &=&
1
+(157388+16a)y^{20}
+(3125056-64a)y^{22}
\\&&
+(52740406-160a)y^{24}
+ \cdots, \\
W_1(S) &=&
y^4
+(-2002+a)y^{16}
+(428099-20a)y^{20}
+ \cdots,\\
\end{array}\right. \\
&\left\{
\begin{array}{lll}
W_2(C) &=&
1
+(157388+16a)y^{20}
+(3431232+1024b-64a)y^{22}
\\&&
+(48040246-10240b-160a)y^{24}
+ \cdots,\\
W_2(S) &=&
y^8
+(-24-b)y^{12}
+(276+22b+a)y^{16}
\\&&
+(394680-231b-20a)y^{20}
+ \cdots,\\
\end{array}\right. \\
&\left\{
\begin{array}{lll}
W_3(C) &=&
1
+(157388+16a)y^{20}
+(3431232+1024b-64a)y^{22}
\\&&
+(47974710-10240b-160a)y^{24}
+ \cdots,\\
W_3(S) &=&
-by^{12}
+(22b+a)y^{16}
+(396704-231b-20a)y^{20}
+ \cdots,\\
\end{array}\right.
\end{align*}
where $a,b$ are integers.
Currently, it is not known whether there is an extremal
singly even self-dual code of length $112$.

\bigskip
\noindent
{\bf Acknowledgment.}
This work was supported by JSPS KAKENHI Grant Number 15H03633.


\end{document}